\documentclass[10pt,a4paper]{article}

\usepackage[T1]{fontenc}
\usepackage{lmodern}
\usepackage{amsmath,amssymb,amsfonts,mathtools}
\usepackage{amsthm}
\usepackage{microtype}
\usepackage[a4paper,left=38mm,right=38mm,top=26mm,bottom=75mm]{geometry}
\usepackage{titlesec}
\usepackage{hyperref}
\hypersetup{colorlinks=true,linkcolor=blue,citecolor=blue,urlcolor=blue}
\titleformat{\section}{\normalfont\Large\bfseries}{\thesection}{0.65em}{}
\titlespacing*{\section}{0pt}{2.3ex plus .7ex minus .2ex}{1.2ex plus .2ex}
\newcommand{\dd}{\,d}

\theoremstyle{plain}
\newtheorem{theorem}{Theorem}[section]
\newtheorem{proposition}[theorem]{Proposition}
\newtheorem{lemma}[theorem]{Lemma}
\newtheorem{corollary}[theorem]{Corollary}
\theoremstyle{remark}
\newtheorem{remark}[theorem]{Remark}

\numberwithin{equation}{section}
\begin{document}

\begin{center}
\vspace*{14mm}
{\LARGE The Fefferman--Szeg\H{o} Sphericity Criterion in\\[3mm]
Complex Dimension Three\par}

\vspace{12mm}
{\large Venkata Siddharth Pendyala\textsuperscript{1*}\par}
\vspace{3mm}
{\normalsize \textsuperscript{1*}Bellevue, USA.\par}

\vspace{10mm}
{\normalsize Corresponding author(s). E-mail(s):\\
\href{mailto:venkatasiddharthpendyala@gmail.com}{venkatasiddharthpendyala@gmail.com};\par}
\end{center}

\vspace{7mm}
\begin{center}
\textbf{Abstract}
\end{center}
\noindent
We establish a Fefferman--Szeg\H{o} characterization of local CR sphericity for smoothly bounded strongly pseudoconvex domains in complex dimension three. We derive the boundary expansion of the normalized determinant of the Fefferman--Szeg\H{o} metric and prove that its second-order coefficient is a universal multiple of the squared Chern--Moser curvature. Hence, vanishing of the second-order deviation from the ball model is equivalent to local sphericity. A logarithmic stability theorem for the associated Monge--Amp\`ere determinant controls the remainder and completes the dimension-three case.

\vspace{4mm}

\vspace{4mm}

\vspace{10mm}
\section{Introduction}

Let $\Omega\Subset\mathbb C^{n}$ be a bounded strongly pseudoconvex domain with $C^\infty$ boundary.  Fix
\[
dV=\bigwedge_{j=1}^{n}\frac{dz_j\wedge d\bar z_j}{-2\mathrm i}.
\]
For every positive defining function $s$, the Fefferman surface measure $\sigma_{\mathrm F}$ is normalized by
\begin{equation}\label{eq:Fefferman-measure}
\sigma_{\mathrm F}\wedge ds
   =J[s]^{1/(n+1)}dV
   \qquad\text{on }\partial\Omega,
\end{equation}
where $J$ is the bordered complex Monge--Amp\`ere operator in \eqref{eq:J}; the boundary form determined by \eqref{eq:Fefferman-measure} is independent of the chosen positive defining function.  We use the Hirachi--Komatsu--Nakazawa normalization, for which the diagonal Szeg\H{o} kernel has leading coefficient $(n-1)!/\pi^n$.  We write $\mathsf S_\Omega(z,w)$ for the resulting reproducing kernel and
\[
\mathsf S_\Omega(z)=\mathsf S_\Omega(z,z).
\]
The Fefferman--Szeg\H{o} metric is the K\"ahler metric
\begin{equation}\label{eq:metric-definition}
g^\Omega_{i\bar j}
 =\partial_i\partial_{\bar j}\log\mathsf S_\Omega,
\end{equation}
and its associated scalar function is
\begin{equation}\label{eq:R-definition}
\mathfrak R_\Omega
 =\frac{\det[g^\Omega_{i\bar j}]}
 {\mathsf S_\Omega^{(n+1)/n}}.
\end{equation}
Barrett and Lee introduced the invariant Szeg\H{o} metric and established its transformation law under $T$-biholomorphisms \cite{BL}.  Here a $T$-biholomorphism is a biholomorphism extending smoothly to the boundary for which the fractional complex Jacobian used in the Hardy-space isometry admits a global holomorphic branch.  Under such maps the two Jacobian weights in \eqref{eq:R-definition} cancel, so $\mathfrak R_\Omega$ is an absolute invariant; see Proposition~\ref{prop:invariance}.  Without a global branch, the singular part of the Szeg\H{o} kernel still transforms locally modulo a smooth error \cite{Hir90,HKN}, which is the level of invariance needed for the boundary coefficient computed below.

For $n\ge4$, Bhatnagar and Fan proved that
\[
\mathfrak R_\Omega=C_n+o(\rho^2),
\qquad
C_n=((n-1)!)^{-(n+1)/n}n^n\pi^{n+1},
\]
characterizes local sphericity, and they isolated $n=2,3$ as the remaining dimensions \cite[Theorem~1.4 and Question~1.5]{BF}.  Their smooth determinant calculation gives the second-order coefficient $((3-3n)/n)q$, hence $-2q$ when $n=3$.  The dimension-three analysis therefore turns on the relative logarithmic term $r^3\log r$ in the diagonal Szeg\H{o} expansion.  Its weighted size is preserved by the parabolic frame
\[
r^{1/2}T_1,\quad r^{1/2}T_2,\quad rN,
\]
which is adapted to Fefferman's invariant calculus \cite{Fef79,Gra87b}.

Section~4 proves that this logarithmic perturbation changes the normalized Monge--Amp\`ere determinant by only $O(r^3(1+|\log r|))$.  Section~5 supplies the corresponding smooth expansion with an $O(r^3)$ remainder.  Together they yield a uniform second-order formula whose coefficient is the squared Chern--Moser curvature.

\begin{theorem}\label{thm:main}
Let $\Omega\Subset\mathbb C^3$ be a bounded strongly pseudoconvex domain with $C^\infty$ boundary.  Put
\begin{equation}\label{eq:C3}
C_3=\frac{27\pi^4}{2^{4/3}}.
\end{equation}
Let $r$ be a positive Fefferman defining function in a boundary collar and let $\pi$ denote the collar projection onto $\partial\Omega$.  Then, uniformly as $r\downarrow0$,
\begin{equation}\label{eq:strong-expansion}
\mathfrak R_\Omega
 =C_3-\frac23C_3
 \|\mathcal A\circ\pi\|^2r^2
 +O\!\left(r^3(1+|\log r|)\right),
\end{equation}
where $\mathcal A$ is the Chern--Moser curvature tensor in the coefficient convention of Hirachi--Komatsu--Nakazawa.  Thus, in Chern--Moser normal coordinates and with $1\le i,j,k,l\le2$, its squared norm is the tensor-index sum
\begin{equation}\label{eq:A-normalization}
\|\mathcal A\|^2
 :=\sum_{i,j,k,l=1}^{2}
 |A^0_{ij\bar k\bar l}|^2.
\end{equation}
The components are symmetric in $i,j$ and in $k,l$.  Equivalently, if the bidegree-$(2,2)$ normal-form polynomial is written once in exponent-multiindex notation as
\[
F_{22}(z,\bar z)
 =\sum_{\substack{\mu,\nu\in\mathbb N_0^2\\|\mu|=|\nu|=2}}
 \widehat A^0_{\mu\bar\nu}z^\mu\bar z^\nu,
\]
then
\begin{equation}\label{eq:A-multiindex-conversion}
\|\mathcal A\|^2
 =\sum_{\substack{\mu,\nu\in\mathbb N_0^2\\|\mu|=|\nu|=2}}
 \frac{\mu!\,\nu!}{(2!)^2}
 |\widehat A^0_{\mu\bar\nu}|^2,
\qquad
\mu!=\mu_1!\mu_2!.
\end{equation}

Consequently, for every smooth defining function $\rho$ with $\Omega=\{\rho<0\}$ and every $p\in\partial\Omega$, the limit
\begin{equation}\label{eq:rho-limit}
\lim_{\substack{z\to p\\ z\in\Omega}}
\frac{\mathfrak R_\Omega(z)-C_3}{\rho(z)^2}
=-\frac23C_3\lambda_\rho(p)^2\|\mathcal A(p)\|^2
\end{equation}
exists, where
\begin{equation}\label{eq:lambda-rho}
\lambda_\rho(p)=\lim_{z\to p}\frac{r(z)}{-\rho(z)}>0.
\end{equation}
In particular, the following are equivalent:
\begin{enumerate}
\item $\partial\Omega$ is locally CR-equivalent to the unit sphere;
\item for every $p\in\partial\Omega$,
\[
\mathfrak R_\Omega(z)=C_3+o(\rho(z)^2)
\quad\text{as }z\to p\text{ in }\Omega.
\]
\end{enumerate}
\end{theorem}

The expansion also yields the following pointwise recovery formula.

\begin{corollary}[Boundary recovery of the Chern--Moser curvature]\label{cor:recovery}
For every $p\in\partial\Omega$,
\begin{equation}\label{eq:curvature-recovery}
\|\mathcal A(p)\|^2
=-\frac{3}{2C_3\lambda_\rho(p)^2}
\lim_{z\to p}
\frac{\mathfrak R_\Omega(z)-C_3}{\rho(z)^2}.
\end{equation}
Thus $p$ is CR umbilical if and only if the limit in \eqref{eq:rho-limit} vanishes.  An open boundary set is locally spherical if and only if this limit vanishes at every point of that set.
\end{corollary}

\section{Transformation law and the Monge--Amp\`ere identity}

We begin with the invariance of the metric and its scalar quotient.

\begin{proposition}\label{prop:invariance}
Let $F:\Omega\to\widetilde\Omega$ be a $T$-biholomorphism.  Then
\begin{equation}\label{eq:S-transform}
\mathsf S_\Omega(z)
=|\det F'(z)|^{2n/(n+1)}
\mathsf S_{\widetilde\Omega}(F(z)),
\end{equation}
\begin{equation}\label{eq:g-transform}
g^\Omega=F^*g^{\widetilde\Omega},
\end{equation}
and
\begin{equation}\label{eq:R-transform}
\mathfrak R_\Omega
=\mathfrak R_{\widetilde\Omega}\circ F.
\end{equation}
\end{proposition}

\begin{proof}
The kernel transformation \eqref{eq:S-transform} is the Fefferman--Hardy-space transformation law of Barrett--Lee \cite[Propositions~1--2]{BL}.  Taking logarithms, the extra term is a constant multiple of $\log|\det F'|^2$, which is pluriharmonic because $\det F'$ is holomorphic and nowhere zero.  Hence \eqref{eq:g-transform} follows after applying $\partial\bar\partial$.  In local coordinates,
\[
\det g^\Omega(z)
=|\det F'(z)|^2
\det g^{\widetilde\Omega}(F(z)),
\]
and \eqref{eq:S-transform} gives
\[
\mathsf S_\Omega(z)^{(n+1)/n}
=|\det F'(z)|^2
\mathsf S_{\widetilde\Omega}(F(z))^{(n+1)/n}.
\]
Their quotient is \eqref{eq:R-transform}.
\end{proof}

For a positive smooth function $u$ on a domain in $\mathbb C^n$, define
\begin{equation}\label{eq:J}
J[u]
=(-1)^n
\det
\begin{pmatrix}
u&u_{\bar j}\\[2mm]
u_i&u_{i\bar j}
\end{pmatrix}.
\end{equation}

\begin{lemma}\label{lem:exact-identity}
For every positive $u$,
\begin{equation}\label{eq:block-identity}
\det\bigl(\partial\bar\partial(-\log u)\bigr)
=\frac{J[u]}{u^{n+1}}.
\end{equation}
Consequently, for $u=\mathsf S_\Omega^{-1/n}$,
\begin{equation}\label{eq:R-J}
\mathfrak R_\Omega=n^nJ[u].
\end{equation}
\end{lemma}

\begin{proof}
The Schur complement of the upper-left entry in the bordered matrix gives
\[
J[u]
=(-1)^n u\det\left(u_{i\bar j}-u^{-1}u_i u_{\bar j}\right).
\]
Since
\[
\partial_i\partial_{\bar j}(-\log u)
=-u^{-1}u_{i\bar j}+u^{-2}u_i u_{\bar j},
\]
extracting $-u^{-1}$ from each row proves \eqref{eq:block-identity}.  If $u=\mathsf S_\Omega^{-1/n}$, then $\log\mathsf S_\Omega=-n\log u$ and
\[
\det[g^\Omega_{i\bar j}]
=n^n\frac{J[u]}{u^{n+1}}.
\]
Division by $\mathsf S_\Omega^{(n+1)/n}=u^{-(n+1)}$ proves \eqref{eq:R-J}.
\end{proof}

\begin{lemma}[The ball]\label{lem:ball}
For the unit ball $\mathbb B^3$,
\begin{equation}\label{eq:ball-value}
\mathfrak R_{\mathbb B^3}\equiv C_3.
\end{equation}
\end{lemma}

\begin{proof}
With $r=1-|z|^2$ and $a_0=2/\pi^3$, the normalized diagonal kernel is
\[
\mathsf S_{\mathbb B^3}=a_0r^{-3}.
\]
Thus $g^{\mathbb B^3}=3\,\partial\bar\partial(-\log r)$ and
$\det g^{\mathbb B^3}=27r^{-4}$.  Hence
\[
\mathfrak R_{\mathbb B^3}
=27a_0^{-4/3}
=\frac{27\pi^4}{2^{4/3}}.
\]
\end{proof}

\section{Kernel asymptotics in dimension three}

We now fix $n=3$.  In the smooth amplitude expansions, $F=O(r^k)$ means $F=r^k\widehat F$ with $\widehat F$ smooth up to the boundary, with uniform control of the derivatives used below.  Logarithmic estimates are understood in the indicated parabolic frame.

A positive Fefferman defining function $r$ is an approximate solution of the complex Monge--Amp\`ere equation and satisfies
\begin{equation}\label{eq:Fefferman-r}
J[r]=1+O(r^4),
\end{equation}
where, in the preceding sense, $J[r]-1$ is smoothly divisible by $r^4$.  Such defining functions and their transformation properties follow from Fefferman's construction and the higher boundary analysis of Lee--Melrose and Graham \cite{Fef76,LM,Gra87a}.

The diagonal Fefferman--Szeg\H{o} kernel has the boundary form
\begin{equation}\label{eq:S-expansion}
\mathsf S_\Omega
=\frac{\varphi}{r^3}+\psi\log r
=\frac{h}{r^3},
\qquad
h=\varphi+\psi r^3\log r,
\end{equation}
with $\varphi,\psi\in C^\infty$ up to the boundary \cite{BdMS,Fef74,HKN}.  After absorbing the smooth remainder into $\varphi$, equation \eqref{eq:S-expansion} holds in a fixed boundary collar.  In our normalization,
\begin{equation}\label{eq:phi-leading}
\varphi|_{\partial\Omega}=a_0,
\qquad
a_0=\frac{2}{\pi^3}.
\end{equation}

\begin{lemma}[Dimension-three Szeg\H{o} amplitude]\label{lem:HKN-normalization}
In the normalization fixed above,
\begin{equation}\label{eq:HKN}
\varphi
=a_0\left(1+\frac13\|\mathcal A\|^2r^2\right)
+r^3\varphi_3,
\qquad
a_0=\frac{2}{\pi^3},
\end{equation}
where $\varphi_3$ is smooth and
\begin{equation}\label{eq:HKN-norm-convention}
\|\mathcal A\|^2
=\sum_{i,j,k,l=1}^{2}|A^0_{ij\bar k\bar l}|^2.
\end{equation}
The indices are ordered, with symmetry in $i,j$ and in $k,l$.  If the same bidegree-$(2,2)$ polynomial is written in exponent-multiindex form,
\[
F_{22}(z,\bar z)
=\sum_{|\mu|=|\nu|=2}
\widehat A^0_{\mu\bar\nu}z^\mu\bar z^\nu,
\]
then
\begin{equation}\label{eq:HKN-multiindex-conversion}
\|\mathcal A\|^2
=\sum_{|\mu|=|\nu|=2}
\frac{\mu!\,\nu!}{(2!)^2}
|\widehat A^0_{\mu\bar\nu}|^2.
\end{equation}
\end{lemma}

\begin{proof}
Hirachi--Komatsu--Nakazawa obtain
\[
\varphi^S
=1+c_n^S\|A^0_{22}\|^2r^2+O(r^3),
\qquad
(n-1)(n-2)c_n^S=\frac23.
\]
Thus $c_3^S=1/3$, and multiplication by the leading factor $(3-1)!/\pi^3=2/\pi^3$ gives \eqref{eq:HKN}.  For a multiindex $\mu$ of length two, the number of ordered strings with content $\mu$ is $2!/\mu!$; grouping the ordered-index sum by monomial therefore gives \eqref{eq:HKN-multiindex-conversion}.
\end{proof}

\begin{remark}[Normalization conventions]\label{rem:HKN-conventions}
The surface measure, bordered Monge--Amp\`ere operator, and Fefferman defining function used here coincide with those of Hirachi--Komatsu--Nakazawa.  Their formulas are \eqref{eq:Fefferman-measure}, \eqref{eq:J}, and \eqref{eq:Fefferman-r}, and their kernel carries the leading factor $(n-1)!/\pi^n$ \cite[Eq.~(1.1), Eqs.~(1.2)--(1.4), Proposition~1, Eq.~(1.7), and Theorem~1]{HKN}.
\end{remark}

Set
\begin{equation}\label{eq:Q-def}
Q=\frac{h}{a_0}.
\end{equation}
After absorbing all smooth terms divisible by $r^3$ into a smooth coefficient, equations \eqref{eq:S-expansion}--\eqref{eq:HKN} give
\begin{equation}\label{eq:Q-expansion}
Q
=1+qr^2+r^3(\widetilde q+\vartheta\log r),
\end{equation}
where $q,\widetilde q,\vartheta$ are smooth real-valued functions and
\begin{equation}\label{eq:q-boundary}
q|_{\partial\Omega}=\frac13\|\mathcal A\|^2.
\end{equation}

Introduce the normalized Szeg\H{o} defining function
\begin{equation}\label{eq:Phi}
\Phi
=\left(\frac{C_3}{27}\mathsf S_\Omega^{4/3}\right)^{-1/4}.
\end{equation}
Since $C_3/27=a_0^{-4/3}$, equations \eqref{eq:S-expansion} and \eqref{eq:Q-def} yield
\begin{equation}\label{eq:Phi-rQ}
\Phi=rQ^{-1/3}.
\end{equation}
Homogeneity of $J$ and Lemma~\ref{lem:exact-identity} give
\begin{equation}\label{eq:J-Phi-R}
J[\Phi]=\frac{\mathfrak R_\Omega}{C_3}.
\end{equation}
Thus the theorem reduces to a second-order expansion of $J[rQ^{-1/3}]$.

\section{Weighted stability of logarithmic perturbations}

We begin with an algebraic identity that removes all first derivatives of the amplitude.

\begin{lemma}\label{lem:multiplication}
Let $r>0$, $Q>0$, and $\Phi=rQ^{-1/3}$.  Then
\begin{equation}\label{eq:exact-Q}
J[\Phi]
=-Q^{-4/3}
\det
\begin{pmatrix}
r&r_{\bar\beta}\\[1mm]
r_\alpha&r_{\alpha\bar\beta}
-\dfrac13r(\log Q)_{\alpha\bar\beta}
\end{pmatrix}.
\end{equation}
\end{lemma}

\begin{proof}
Write $f=Q^{-1/3}$ and $L=\log f$.  Since $\Phi=rf$,
\[
\Phi_\alpha=f(r_\alpha+rL_\alpha)
\]
and
\[
\Phi_{\alpha\bar\beta}
=f\left[r_{\alpha\bar\beta}+r_\alpha L_{\bar\beta}
+r_{\bar\beta}L_\alpha
+r(L_{\alpha\bar\beta}+L_\alpha L_{\bar\beta})\right].
\]
Extract $f$ from each of the four rows of the bordered determinant, subtract $L_\alpha$ times the first row from row $\alpha+1$, and subtract $L_{\bar\beta}$ times the first column from column $\bar\beta+1$.  Every first derivative of $L$ cancels, leaving
\[
J[\Phi]
=-f^4\det
\begin{pmatrix}
r&r_{\bar\beta}\\
r_\alpha&r_{\alpha\bar\beta}+rL_{\alpha\bar\beta}
\end{pmatrix}.
\]
Now $f^4=Q^{-4/3}$ and $L_{\alpha\bar\beta}=-(1/3)(\log Q)_{\alpha\bar\beta}$.
\end{proof}

The next result is stated for general logarithmic order because it identifies the precise dimensional threshold.

\begin{theorem}[Dimension-three parabolic stability for logarithmic amplitudes]\label{thm:log-stability}
Let $m\ge2$ be an integer.  In a boundary collar suppose
\begin{equation}\label{eq:general-log-Q}
Q=Q_0+r^m\vartheta\log r,
\end{equation}
where $Q_0$ and $\vartheta$ are smooth, $Q_0>0$, and
\[
Q_0-1=r^2b_0
\]
for a smooth function $b_0$ on the closed collar.  Put
\[
\Phi=rQ^{-1/3},
\qquad
\Phi_0=rQ_0^{-1/3}.
\]
Then, uniformly in the collar,
\begin{equation}\label{eq:general-log-bound}
J[\Phi]-J[\Phi_0]
=O\!\left(r^m(1+|\log r|)\right).
\end{equation}
In particular, when $m=3$,
\begin{equation}\label{eq:log-negligible}
J[\Phi]-J[\Phi_0]
=O\!\left(r^3(1+|\log r|)\right)
=o(r^2).
\end{equation}
\end{theorem}

\begin{proof}
All estimates are local and uniform; a finite cover of the compact boundary gives the global collar statement.  Shrink the collar so that $0<r<e^{-1}$ and $Q_0$ is bounded away from zero.  Choose smooth $(1,0)$ fields $T_1,T_2$ spanning the complex tangent bundle of the level sets of $r$, and choose a transverse field $N$ with $Nr=1$.  For example, relative to Euclidean coordinates one may take
\[
N=\frac{1}{|\partial r|^2}
\sum_{j=1}^3r_{\bar j}\frac{\partial}{\partial z_j},
\]
then complete it to a smooth local frame.  Since $\dd r\ne0$ on the boundary, the frame coefficients and all derivatives needed below are uniformly bounded after shrinking the collar.

Define the pointwise parabolic frame
\begin{equation}\label{eq:scaled-frame}
Y_a=r^{1/2}T_a\quad(a=1,2),
\qquad
Y_0=rN,
\end{equation}
and put $\Lambda(r)=1+|\log r|$.  Let
\[
f_m(r)=r^m\log r,
\qquad
a=\frac{\vartheta}{Q_0},
\qquad
\nu=f_m(r)a.
\]
Then $Q/Q_0=1+\nu$ and
\begin{equation}\label{eq:sigma-def}
\sigma:=\log\frac{\Phi}{\Phi_0}
=-\frac13\log(1+\nu).
\end{equation}
After another shrinking, $|\nu|\le1/2$.  Direct differentiation gives
\begin{equation}\label{eq:fm-bounds}
|f_m|\le Cr^m\Lambda,
\qquad
|f_m'|\le Cr^{m-1}\Lambda,
\qquad
|f_m''|\le Cr^{m-2}\Lambda.
\end{equation}
As an identity of $(1,1)$-forms,
\begin{align}
\partial\bar\partial\nu
={}&f_m''a\,\partial r\otimes\bar\partial r
+f_m'a\,\partial\bar\partial r \notag\\
&+f_m'\bigl(\partial a\otimes\bar\partial r
+\partial r\otimes\bar\partial a\bigr)
+f_m\,\partial\bar\partial a,
\label{eq:nu-hessian}
\end{align}
and
\begin{equation}\label{eq:nu-first}
\partial\nu=f_m'a\,\partial r+f_m\,\partial a.
\end{equation}
Since $T_ar=0$ and $Nr=1$, equations \eqref{eq:fm-bounds}--\eqref{eq:nu-first} imply
\begin{align}
(\partial\bar\partial\nu)(T_a,\overline T_b)
&=O(r^{m-1}\Lambda),\label{eq:nu-TT}\\
(\partial\bar\partial\nu)(T_a,\overline N)
&=O(r^{m-1}\Lambda),\label{eq:nu-TN}\\
(\partial\bar\partial\nu)(N,\overline N)
&=O(r^{m-2}\Lambda),\label{eq:nu-NN}
\end{align}
and
\[
(\partial\nu)(T_a)=O(r^m\Lambda),
\qquad
(\partial\nu)(N)=O(r^{m-1}\Lambda).
\]
Furthermore,
\begin{equation}\label{eq:sigma-hessian}
\partial\bar\partial\sigma
=-\frac13\left(
\frac{\partial\bar\partial\nu}{1+\nu}
-\frac{\partial\nu\otimes\bar\partial\nu}{(1+\nu)^2}
\right).
\end{equation}
The first-derivative estimates yield, in the scaled frame,
\begin{align}
(\partial\nu\otimes\bar\partial\nu)(Y_a,\overline Y_b)
 &=O(r^{2m+1}\Lambda^2),\label{eq:quadratic-TT}\\
(\partial\nu\otimes\bar\partial\nu)(Y_a,\overline Y_0)
 &=O(r^{2m+1/2}\Lambda^2),\label{eq:quadratic-TN}\\
(\partial\nu\otimes\bar\partial\nu)(Y_0,\overline Y_0)
 &=O(r^{2m}\Lambda^2).\label{eq:quadratic-NN}
\end{align}
Because $m\ge2$ and $0<r<e^{-1}$, every quantity in \eqref{eq:quadratic-TT}--\eqref{eq:quadratic-NN} is $O(r^m\Lambda)$.  The linear Hessian term in \eqref{eq:sigma-hessian}, after insertion of the scaling factors from \eqref{eq:scaled-frame}, satisfies
\begin{align*}
(\partial\bar\partial\nu)(Y_a,\overline Y_b)
 &=O(r^m\Lambda),\\
(\partial\bar\partial\nu)(Y_a,\overline Y_0)
 &=O(r^{m+1/2}\Lambda),\\
(\partial\bar\partial\nu)(Y_0,\overline Y_0)
 &=O(r^m\Lambda).
\end{align*}
Since $(1+\nu)^{-1}$ and $(1+\nu)^{-2}$ are uniformly bounded, equation \eqref{eq:sigma-hessian} gives
\begin{equation}\label{eq:sigma-scaled}
(\partial\bar\partial\sigma)(Y_A,\overline Y_B)
=O(r^m\Lambda),
\qquad 0\le A,B\le2,
\end{equation}
with the sharper mixed estimate $O(r^{m+1/2}\Lambda)$.

Set
\[
G_0=\partial\bar\partial(-\log\Phi_0).
\]
Because $Q_0-1=r^2b_0$ with $b_0$ smooth,
\begin{equation}\label{eq:G0-splitting}
G_0
=\frac{\partial r\otimes\bar\partial r}{r^2}
-\frac{\partial\bar\partial r}{r}
+\frac13\partial\bar\partial\log Q_0.
\end{equation}
In the scaled frame, the last term tends to zero.  If $\mathcal L_{a\bar b}=-\partial\bar\partial r(T_a,\overline T_b)|_{r=0}$ is the positive Levi matrix in our sign convention, then
\begin{align}
G_0(Y_a,\overline Y_b)
&=\mathcal L_{a\bar b}+O(r),\label{eq:G0-TT}\\
G_0(Y_a,\overline Y_0)
&=O(r^{1/2}),\label{eq:G0-TN}\\
G_0(Y_0,\overline Y_0)
&=1+O(r).\label{eq:G0-NN}
\end{align}
Strong pseudoconvexity supplies constants $c,C>0$ such that the matrix
\[
M_0=\bigl(G_0(Y_A,\overline Y_B)\bigr)_{A,B=0}^2
\]
satisfies
\begin{equation}\label{eq:M0-bounds}
cI\le M_0\le CI,
\qquad
\|M_0^{-1}\|\le C.
\end{equation}

Since $-\log\Phi=-\log\Phi_0-\sigma$, the perturbed scaled matrix is
\[
M=M_0-E,
\qquad
E_{A\bar B}=(\partial\bar\partial\sigma)(Y_A,\overline Y_B).
\]
By \eqref{eq:sigma-scaled}, $\|E\|=O(r^m\Lambda)$; hence
\begin{equation}\label{eq:det-ratio-log}
\frac{\det M}{\det M_0}
=\det(I-M_0^{-1}E)
=1+O(r^m\Lambda).
\end{equation}
This determinant ratio is independent of the common frame: if $B$ is the frame-change matrix, then $M=B^*GB$ and the factor $|\det B|^2$ cancels between numerator and denominator.  Also,
\begin{equation}\label{eq:phi-ratio-log}
\left(\frac{\Phi}{\Phi_0}\right)^4
=(1+\nu)^{-4/3}
=1+O(r^m\Lambda).
\end{equation}
The lower bound in \eqref{eq:M0-bounds} implies $\det M_0\ge c^3$.  Since $\Phi_0>0$, equation \eqref{eq:block-identity} gives $J[\Phi_0]>0$, so the following quotient is well defined.  By Lemma~\ref{lem:exact-identity} in dimension three,
\begin{equation}\label{eq:J-ratio-log}
\frac{J[\Phi]}{J[\Phi_0]}
=\left(\frac{\Phi}{\Phi_0}\right)^4
\frac{\det M}{\det M_0}
=1+O(r^m\Lambda).
\end{equation}
Moreover, $J[\Phi_0]$ is uniformly bounded above and below.  Let $X=(T_1,T_2,N)$, let $C_X$ be its coefficient matrix relative to the coordinate frame, and put $D_r=\operatorname{diag}(r^{1/2},r^{1/2},r)$.  If $G_{0,\mathrm{coord}}$ denotes the coordinate matrix of $G_0$, then
\[
M_0=D_r^*C_X^*G_{0,\mathrm{coord}}C_XD_r,
\]
and therefore, using $\Phi_0^4=r^4Q_0^{-4/3}$,
\begin{equation}\label{eq:J0-two-sided}
J[\Phi_0]
=\Phi_0^4\det G_{0,\mathrm{coord}}
=Q_0^{-4/3}\frac{\det M_0}{|\det C_X|^2}.
\end{equation}
The functions $Q_0$, $|\det C_X|$, and their reciprocals are uniformly bounded on a reduced closed collar, while \eqref{eq:M0-bounds} bounds $\det M_0$ above and below.  Hence constants $0<c_0<C_0$ exist such that
\[
c_0\le J[\Phi_0]\le C_0.
\]
Multiplication of \eqref{eq:J-ratio-log} by this two-sided bound proves \eqref{eq:general-log-bound}.
\end{proof}

\begin{remark}\label{rem:threshold}
Theorem~\ref{thm:log-stability} exhibits the dimensional threshold directly: a relative logarithmic term of order $r^m\log r$ retains parabolic order $m$ through the two derivatives entering the metric determinant, so that $m=3$ contributes $o(r^2)$ and $m=2$ contributes at the critical second order.
\end{remark}

\section{The smooth determinant calculation}

The next proposition gives the smooth second-order expansion with a uniform $O(r^3)$ remainder.  Its coefficient is the specialization to $n=3$ of the determinant computation in Bhatnagar--Fan \cite[Lemma~5.5]{BF}.

\begin{proposition}\label{prop:J-expansion}
Let $r$ satisfy \eqref{eq:Fefferman-r}, let
\begin{equation}\label{eq:Q0}
Q_0=1+qr^2+\widetilde q r^3
\end{equation}
with $q,\widetilde q$ smooth, and put $\Phi_0=rQ_0^{-1/3}$.  Then, uniformly in a boundary collar,
\begin{equation}\label{eq:J-second-order}
J[\Phi_0]
=1-2qr^2+O(r^3).
\end{equation}
\end{proposition}

\begin{proof}
Taylor expansion gives
\begin{equation}\label{eq:Q-power-smooth}
Q_0^{-4/3}=1-\frac43qr^2+O(r^3)
\end{equation}
and
\begin{equation}\label{eq:logQ-smooth}
\log Q_0=qr^2+\widetilde q r^3+r^4b
\end{equation}
for a smooth function $b$.  Direct differentiation gives
\begin{align}
(\log Q_0)_{\alpha\bar\beta}
={}&2q r_\alpha r_{\bar\beta}
+2r\bigl(q_\alpha r_{\bar\beta}
+q_{\bar\beta}r_\alpha
+q r_{\alpha\bar\beta}
+3\widetilde q r_\alpha r_{\bar\beta}\bigr)
+r^2E_{\alpha\bar\beta},
\label{eq:logQ-Hessian}
\end{align}
where
\begin{align}
E_{\alpha\bar\beta}
={}&q_{\alpha\bar\beta}
+3\bigl(\widetilde q_\alpha r_{\bar\beta}
+\widetilde q_{\bar\beta}r_\alpha
+\widetilde q r_{\alpha\bar\beta}\bigr)
+r\widetilde q_{\alpha\bar\beta}\notag\\
&+12b r_\alpha r_{\bar\beta}
+4r\bigl(b_\alpha r_{\bar\beta}
+b_{\bar\beta}r_\alpha
+b r_{\alpha\bar\beta}\bigr)
+r^2b_{\alpha\bar\beta}.
\label{eq:E-explicit}
\end{align}
All coefficients in \eqref{eq:E-explicit} are uniformly bounded on a sufficiently small collar.

Insert \eqref{eq:logQ-Hessian} into the bordered matrix in Lemma~\ref{lem:multiplication}.  For each $\alpha$, add
\begin{equation}\label{eq:row-operation-coefficient}
c_\alpha
=\frac23\left(qr_\alpha r+
(q_\alpha+3\widetilde q r_\alpha)r^2\right)
\end{equation}
times the first row to row $\alpha+1$, and, for each $\beta$, add
\begin{equation}\label{eq:column-operation-coefficient}
d_{\bar\beta}=\frac23q_{\bar\beta}r^2
\end{equation}
times the first column to column $\bar\beta+1$.  The row operations cancel the terms containing $qr r_\alpha r_{\bar\beta}$, $q_\alpha r^2r_{\bar\beta}$, and $\widetilde q r^2r_\alpha r_{\bar\beta}$, and the column operations cancel the remaining terms $q_{\bar\beta}r^2r_\alpha$.  The resulting matrix has the form
\begin{equation}\label{eq:transformed-bordered-matrix}
\begin{pmatrix}
r&r_{\bar\beta}\\
r_\alpha(1+\frac23qr^2)&
r_{\alpha\bar\beta}(1-\frac23qr^2)
\end{pmatrix}
+\mathcal E,
\end{equation}
where the entries of $\mathcal E$ may be read directly from the preceding operations: the upper-left entry is zero, the upper-right entries are
\[
\mathcal E_{0\bar\beta}=\frac23q_{\bar\beta}r^3,
\]
the lower-left entries are
\[
\mathcal E_{\alpha\bar0}
=\frac23(q_\alpha+3\widetilde q r_\alpha)r^3,
\]
and the lower-right entries are
\[
\mathcal E_{\alpha\bar\beta}
=-\frac13r^3E_{\alpha\bar\beta}+O(r^4).
\]
Consequently every entry of $\mathcal E$ is uniformly $O(r^3)$.

Set
\[
A=1+\frac23qr^2,
\qquad
B=1-\frac23qr^2.
\]
On each coordinate chart of a finite collar cover, the entries of $r$, its first two derivatives, and the smooth functions occurring above are uniformly bounded.  Since $A$ and $B$ remain bounded away from zero after the collar is reduced, multilinearity of the determinant and the entrywise estimate $\mathcal E=O(r^3)$ yield
\begin{align}
&\det\left[
\begin{pmatrix}
r&r_{\bar\beta}\\
Ar_\alpha&Br_{\alpha\bar\beta}
\end{pmatrix}
+\mathcal E\right]\notag\\
&\qquad
=AB^2
\det\begin{pmatrix}
r(B/A)&r_{\bar\beta}\\
r_\alpha&r_{\alpha\bar\beta}
\end{pmatrix}
+O(r^3).
\label{eq:factored-bordered-matrix}
\end{align}
Here the factor $AB^2$ follows by extracting $A$ from the first column and $B$ from the three remaining columns, and then multiplying the first row of the residual determinant by $B$.  Moreover,
\[
\frac BA=1-\frac43qr^2+O(r^4),
\]
so the upper-left entry $r(B/A)$ differs from $r$ by $O(r^3)$; another application of determinant multilinearity gives
\begin{equation}\label{eq:bordered-comparison}
\det\begin{pmatrix}
r(B/A)&r_{\bar\beta}\\
r_\alpha&r_{\alpha\bar\beta}
\end{pmatrix}
=
\det\begin{pmatrix}
r&r_{\bar\beta}\\
r_\alpha&r_{\alpha\bar\beta}
\end{pmatrix}
+O(r^3).
\end{equation}
Combining Lemma~\ref{lem:multiplication}, \eqref{eq:Q-power-smooth}, \eqref{eq:factored-bordered-matrix}, and \eqref{eq:bordered-comparison}, and using the definition of $J[r]$, we obtain
\begin{equation}\label{eq:factorization}
J[\Phi_0]
=\left(1-\frac43qr^2\right)
\left(1+\frac23qr^2\right)
\left(1-\frac23qr^2\right)^2
J[r]+O(r^3).
\end{equation}
The coefficient of $qr^2$ in this product equals
\[
-\frac43+\frac23-2\cdot\frac23=-2,
\]
and the Fefferman equation \eqref{eq:Fefferman-r}, namely $J[r]=1+O(r^4)$, completes the proof of \eqref{eq:J-second-order}.
\end{proof}

\begin{corollary}\label{cor:exact-expansion}
For $Q$ in \eqref{eq:Q-expansion} and $\Phi=rQ^{-1/3}$,
\begin{equation}\label{eq:JPhi-final}
J[\Phi]
=1-2qr^2+O\!\left(r^3(1+|\log r|)\right).
\end{equation}
Consequently,
\begin{equation}\label{eq:R-final}
\mathfrak R_\Omega
=C_3-2C_3qr^2
+O\!\left(r^3(1+|\log r|)\right).
\end{equation}
\end{corollary}

\begin{proof}
Apply Theorem~\ref{thm:log-stability} with $m=3$ to remove the logarithmic term, then use Proposition~\ref{prop:J-expansion}.  Equation \eqref{eq:R-final} follows from \eqref{eq:J-Phi-R}.
\end{proof}

\section{Proof of the characterization}

\begin{proof}[Proof of Theorem~\ref{thm:main}]
Because $q$ is smooth,
\[
q(z)=q(\pi(z))+O(r(z)).
\]
Substituting \eqref{eq:q-boundary} into \eqref{eq:R-final} therefore gives the uniform expansion
\[
\mathfrak R_\Omega(z)
=C_3-\frac23C_3
\|\mathcal A(\pi(z))\|^2r(z)^2
+O\!\left(r(z)^3(1+|\log r(z)|)\right),
\]
which is \eqref{eq:strong-expansion}.

Let $\rho$ be any smooth defining function with $\Omega=\{\rho<0\}$.  Since $r$ and $-\rho$ are positive defining functions for the same boundary, there is a smooth real function $\omega$ in a collar such that
\begin{equation}\label{eq:defining-comparison}
r=e^\omega(-\rho).
\end{equation}
Thus $\lambda_\rho(p)=e^{\omega(p)}$, and division of \eqref{eq:strong-expansion} by $\rho^2$ gives \eqref{eq:rho-limit}; the remainder tends to zero because
\[
\frac{r^3(1+|\log r|)}{\rho^2}
=e^{2\omega}r(1+|\log r|)\longrightarrow0.
\]

The coefficient in \eqref{eq:rho-limit} is nonzero and $\lambda_\rho(p)>0$, and therefore the limit vanishes at every boundary point exactly when
\[
\mathcal A\equiv0\quad\text{on }\partial\Omega.
\]
The Chern--Moser equivalence theorem states that, for a strongly pseudoconvex hypersurface of CR dimension at least two, the vanishing of the Chern--Moser curvature tensor on a neighbourhood of a point is equivalent to local CR flatness there, and local CR flatness is equivalent to local CR equivalence with the unit sphere \cite{CM}.  The hypothesis at every boundary point therefore implies that the smooth tensor $\mathcal A$ vanishes identically on $\partial\Omega$; it consequently vanishes on a neighbourhood of each point within $\partial\Omega$, and the Chern--Moser equivalence theorem gives local sphericity throughout the boundary.
\end{proof}

\begin{proof}[Proof of Corollary~\ref{cor:recovery}]
Equation \eqref{eq:curvature-recovery} is a rearrangement of \eqref{eq:rho-limit}.  A point is CR umbilical exactly when $\mathcal A(p)=0$.  On an open set, the tensor vanishes identically exactly when the Chern--Moser normal form is locally spherical at every point.
\end{proof}

\begin{remark}[The logarithmic order in dimension two]\label{rem:n2}
In complex dimension two,
\[
\mathsf S_\Omega
=\frac{\varphi}{r^2}+\psi\log r
=\frac{\varphi+\psi r^2\log r}{r^2}.
\]
The normalized amplitude therefore contains a logarithmic perturbation of order $m=2$.  Theorem~\ref{thm:log-stability} bounds its determinant contribution by $O(r^2(1+|\log r|))$, so it can contribute at the second boundary order.  Further information about the logarithmic coefficient is required to decide the corresponding criterion in that dimension.
\end{remark}

\section*{Statements and Declarations}

\textbf{Funding.} The author received no external funding for this work.

\textbf{Competing interests.} The author declares that there are no financial or non-financial interests directly or indirectly related to the work submitted for publication.

\textbf{Data availability.} No datasets were generated or analysed in the course of this theoretical work.

\textbf{Author contributions.} Venkata Siddharth Pendyala is the sole author and is responsible for the conception, proof, exposition, and final approval of the manuscript.

\end{document}